\documentclass[11pt, reqno]{amsart}
\usepackage[margin=1.1in]{geometry}
\usepackage{amsfonts,amsthm,amsmath,amssymb}
\usepackage{amscd}
\usepackage{graphicx}
\usepackage{subfig}
\usepackage{caption}
\usepackage{textcomp} 
\usepackage{tikz,color,graphicx}
\usepackage{pgfplots}
\usepackage{floatflt,epsfig}
\usepackage{bbm}
\usepackage{tikz}
\usetikzlibrary{positioning,chains,fit,shapes,calc}
\usetikzlibrary{arrows,intersections}

\usepackage{hyperref} 
\hypersetup{colorlinks=true, linkcolor=blue, citecolor=blue, filecolor=blue, urlcolor=blue}
\usepackage[comma, sort&compress, numbers]{natbib}

\newcommand{\R}{{\mathbb R}}
\newcommand{\N}{{\mathbb N}}

\newcommand{\eps}{\epsilon}
\def\d{\mathrm{d}}
\def\e{\mathrm{e}} 
\def\N{\mathbb{N}}
\def\R{\mathbb{R}}
\newcommand{\la}{\langle}
\newcommand{\ra}{\rangle}

\def\dfrac#1#2{\lower0.15ex\hbox{\large$\frac{#1}{#2}$}}

\numberwithin{equation}{section}

\theoremstyle{plain} 
\newtheorem{theorem}{Theorem}[section] 
\newtheorem{lemma}[theorem]{Lemma} 
\newtheorem{proposition}[theorem]{Proposition} 
\newtheorem{remark}[theorem]{Remark} 
\newtheorem{definition}[theorem]{Definition}
\newtheorem{corollary}[theorem]{Corollary}

\numberwithin{equation}{section}

\begin{document} 

\title[LDP for Inhomogeneous Erd\H{o}s-R\'enyi random graphs]{Large deviation principle for the maximal eigenvalue of inhomogeneous Erd\H{o}s-R\'enyi random graphs}

\author[A. Chakrabarty]{Arijit Chakrabarty}
\author[R.~S.~Hazra]{Rajat Subhra Hazra}
\address{Indian Statistical Institute, 203 B.T. Road, Kolkata 700108, India}
\email{arijit.isi@gmail.com}
\email{rajatmaths@gmail.com}
\author[F.\ den Hollander]{Frank den Hollander}
\author[M.\ Sfragara]{Matteo Sfragara} 
\address{Mathematisch instituut, Universiteit Leiden, The Netherlands}
\email{denholla@math.leidenuniv.nl}
\email{m.sfragara@math.leidenuniv.nl}

\date{\today}

\begin{abstract}

We consider an inhomogeneous Erd\H{o}s-R\'enyi random graph $G_N$ with vertex set $[N] = \{1,\dots,N\}$ for which the pair of vertices $i,j \in [N]$, $i\neq j$, is connected by an edge with probability $r(\tfrac{i}{N},\tfrac{j}{N})$, independently of other pairs of vertices. Here, $r\colon\,[0,1]^2 \to (0,1)$ is a symmetric function that plays the role of a reference graphon. Let $\lambda_N$ be the maximal eigenvalue of the adjacency matrix of $G_N$. It is known that $\lambda_N/N$ satisfies a large deviation principle as $N \to \infty$. The associated rate function $\psi_r$ is given by a variational formula that involves the rate function $I_r$ of a large deviation principle on graphon space. We analyse this variational formula in order to identify the properties of $\psi_r$, specially when the reference graphon is of rank 1.
\end{abstract}


\maketitle



\section{Introduction and main results}


\subsection{Motivation}

In the past 20 years, many properties have been derived about spectra of random matrices associated with \emph{random graphs}, like the adjacency matrix and the Laplacian matrix \cite{bauer2001random, bhamidi2012spectra, bordenave2010resolvent, dembo:eyal, ding2010spectral, dumitriu2012sparse, farkas2001spectra, jiang2012empirical, jiang2012low, khorunzhy2004eigenvalue, lee:schnelli, tran2013sparse, zhu:2018}. The focus of the present paper is on \emph{inhomogeneous} Erd\H{o}s-R\'enyi random graphs, which are rooted in the theory of complex networks. We consider the \emph{dense regime}, where the average degree of the vertices are proportional to the size of the graph, and analyse the rate function of the \emph{large deviation principle for the maximal eigenvalue of the adjacency matrix} derived in \cite{souvik2020}. In \cite{chakrabarty2019}  the \emph{non-dense non-sparse regime} was considered, where the degrees diverge but sublinearly in the size of the graph, and identified the scaling limit of the empirical spectral distribution of both the adjacency matrix and the Laplacian matrix. Recent results on the maximal eigenvalue in the \emph{sparse regime}, where the degrees are stochastically bounded, can be found in \cite{benaych2017largest}. 

Large deviations of Erd\H{o}s-R\'enyi random graphs were studied in \cite{chatterjeeln2017, chatterjee2011, lubetzkyzhao2015} with the help of the theory of graphons, in particular, subgraph densities and maximal eigenvalues. We refer to \cite{chatterjeeln2017} for a comprehensive review of the literature. Large deviation theory for random matrices started in \cite{arous:guionnet}, with the study of large deviations of the empirical spectral distribution of $\beta$-ensembles with a quadratic potential. The rate was shown to be the square of the number of vertices, and the rate function was shown to be given by a non-commutative notion of entropy. The maximal eigenvalue for such ensembles was studied in \cite{arous:guionnet:dembo}. Large deviations of the empirical spectral distribution of random matrices with non-Gaussian tails were derived in \cite{bordenavecaputo2014}. More recently, the maximal eigenvalue in that setting was studied in \cite{augeri2016, augeriguionnethusson2019}. The adjacency matrix of an inhomogeneous Erd\H{o}s-R\'enyi random graph does not fall in this regime, and hence different techniques are needed in the present paper.


\subsection{LDP for inhomogeneous Erd\H{o}s-R\'enyi random graphs}

Let
\begin{equation}
\mathcal{W} = \bigl\{h\colon\, [0,1]^2 \to [0,1]\colon\,h(x,y) = h(y,x)\,\, \forall \, x,y \in [0,1] \bigr\}
\end{equation}
denote the set of \emph{graphons}. Let $\mathcal{M}$ be the set of Lebesgue measure-preserving bijective maps $\phi\colon\, [0,1] \mapsto [0,1]$. For two graphons $h_1, h_2 \in \mathcal{W}$, the \emph{cut-distance} is defined by 
\begin{equation}
d_{\square}(h_1, h_2) = \sup_{S,T \subset [0,1]} \bigg| \int_{S \times T}\d x \, \d y \, \big[h_1(x,y) - h_2(x,y) \big]  \bigg|,
\end{equation}
and the \emph{cut-metric} by
\begin{equation}
\delta_{\square}(h_1, h_2) = \inf_{\phi \in \mathcal{M}} d_{\square}(h_1, h_2^{\phi}),
\end{equation}
where $h_2^{\phi}(x,y) = h_2(\phi(x), \phi(y))$. The cut-metric defines an equivalence relation $\sim$ on $\mathcal{W}$ by declaring $h_1 \sim h_2$ if and only if $\delta_{\square}(h_1, h_2)= 0$, and leads to the quotient space $\widetilde{\mathcal{W}} = \mathcal{W}/_\sim$. For $h \in \mathcal{W}$, we write $\widetilde{h}$ to denote the equivalence class of $h$ in $\widetilde{\mathcal{W}}$. The equivalence classes correspond to relabelings of the vertices of the graph. The pair $(\widetilde{\mathcal{W}}, \delta_{\square})$ is a \emph{compact metric space} \cite{lovasz2012}. 

Let $r \in \mathcal{W}$ be a \emph{reference graphon} satisfying
\begin{equation}
\label{Assbasic}
\exists\,\eta>0\colon\qquad \eta \leq r(x,y) \leq 1-\eta \quad \forall\,x,y \in [0,1]^2.
\end{equation}
Fix $N \in \mathbb{N}$ and consider the random graph $G_N$ with vertex set $[N]=\{1, \dots, N \}$ where the pair of vertices $i,j \in [N]$, $i \neq j$, is connected by an edge with probability $r(\tfrac{i}{N}, \tfrac{j}{N})$, independently of other pairs of vertices. Write $\mathbb{P}_N$ to denote the law of $G_N$. Use the same symbol for the law on $\mathcal{W}$ induced by the map that associates with the graph $G_N$ its graphon $h^{G_N}$, defined by
\begin{equation}
h^{G_N}(x,y)
= \left\{\begin{array}{ll}
1, &\text{if there is an edge between vertex $\lceil Nx\rceil$ and vertex $\lceil Ny\rceil$},\\
0, &\text{otherwise}. 
\end{array}
\right. 
\end{equation}
Write $\widetilde{\mathbb{P}}_N$ to denote the law of $\widetilde{h}^{G_N}$.

The following LDP is proved in \cite{souvik2020} and is an extension of the celebrated LDP for homogeneous ERRG derived in \cite{chatterjee2011} and further properties of the rate functions were derived in~\cite{lubetzkyzhao2015}.

\begin{theorem}{\bf [LDP for inhomogeneous ERRG]}
\label{thm:LDPinhom}
Subject to \eqref{Assbasic}, the sequence $(\widetilde{\mathbb{P}}_N)_{N\in\mathbb{N}}$ satisfies the large deviation principle on $(\widetilde{\mathcal{W}},\delta_{\square})$ with rate $\binom{N}{2}$, i.e., 
\begin{equation}
\begin{aligned}
\limsup_{N \to \infty} \frac{1}{\binom{N}{2}} \log \widetilde{\mathbb{P}}_N (\mathcal{C}) 
&\leq - \inf_{\widetilde h \in \mathcal{C}} J_r(\widetilde h)
&\forall\,\mathcal{C} \subset \widetilde{\mathcal{W}} \text{ closed},\\ 
\liminf_{N \to \infty} \frac{1}{\binom{N}{2}} \log \widetilde{\mathbb{P}}_N(\mathcal{O}) 
&\geq - \inf_{\widetilde h \in \mathcal{O}} J_r(\widetilde h)
&\forall\,\mathcal{O} \subset \widetilde{\mathcal{W}} \text{ open},
\end{aligned}
\end{equation}
where the rate function $J_r\colon\,\widetilde{\mathcal{W}} \to \mathbb{R}$ is given by
\begin{equation}
\label{Jrhdef}
J_r(\widetilde h) = \inf_{\phi \in \mathcal{M}} I_r(h^\phi),
\end{equation}
where $h$ is any representative of $\widetilde h$ and 
\begin{equation}
\label{Irhdef}
I_r(h) = \int_{[0,1]^2} \d x\, \d y \,\, \mathcal{R}\big(h(x,y) \mid r(x,y)\big), \quad h \in \mathcal{W},
\end{equation}
with
\begin{equation}
\label{Rdef}
\mathcal{R}\big(a \mid b\big) = a \log \tfrac{a}{b} + (1-a) \log \tfrac{1-a}{1-b}
\end{equation}
the relative entropy of two Bernoulli distributions with success probabilities $a \in [0,1]$, $b \in (0,1)$ (with the convention $0 \log 0 =0$).
\end{theorem}

\noindent
It is clear that $J_r$ is a good rate function, i.e., $J_r \not\equiv \infty$ and $J_r$ has compact level sets. Note that \eqref{Jrhdef} differs from the expression in \cite{souvik2020}, where the rate function is the lower semi-continuous envelope of $I_r(h)$. However, it was shown in \cite{Markering2020} that, under the integrability conditions $\log r, \log(1-r) \in L^1([0,1]^2)$, the two rate functions are equivalent, since $J_r(\widetilde h)$ is lower semi-continuous on $\widetilde{\mathcal{W}}$. Clearly, these integrability conditions are implied by \eqref{Assbasic}.


\subsection{Graphon operators}

With $h \in \mathcal{W}$ we associate a graphon operator acting on $L^2([0,1])$, defined as the linear integral operator
\begin{equation} 
\label{graphonoperator}
(T_h u)(x) = \int_{[0,1]} \d y \, h(x,y) u(y), \qquad x \in [0,1],
\end{equation}
with $u \in L^2([0,1])$. The operator norm of $T_h$ is defined as
\begin{equation}
\| T_h \| = \sup_{ {u \in L^2([0,1])} \atop {\| u \|_2 = 1} } \| T_hu \|_2,
\end{equation}
where $\| \cdot \|_2$ denotes the $L^2$-norm. Given a graphon $h \in \mathcal{W}$, we have $\|T_h \| \leq \| h \|_2$. Hence, a graphon sequence converging in the $L^2$-norm also converges in the operator norm. 

The product of two graphons $h_1,h_2 \in \mathcal{W}$ is defined as
\begin{equation}
(h_1 h_2)(x,y) = \int_{[0,1]}\d z \, h_1(x,z)h_2(z,y), \qquad (x,y) \in [0,1]^2,
\end{equation}
and the $n$-th power of a graphon $h \in \mathcal{W}$ as
\begin{equation}
h^n(x,y) = \int_{[0,1]^{n-1}} \d z_1 \cdots  \d z_{n-1} \, h(x,z_1) \times \cdots \times h(z_{n-1}, y), 
\qquad (x,y) \in [0,1]^2,\,n \in \mathbb{N}.
\end{equation}

\begin{definition}{\bf [Eigenvalues and eigenfunctions]}
$\mu \in \mathbb{R}$ is said to be an eigenvalue of the graphon operator $T_h$ if there exists a non-zero function $u \in L^2([0,1])$ such that
\begin{equation}
(T_hu)(x) = \mu u(x), \qquad x \in [0,1].
\end{equation}
The function $u$ is said to be an eigenfunction associated with $\mu$.
\end{definition}

\begin{proposition}{\bf [Properties of the graphon operator]}
\label{prop:basics}
For any $h \in \mathcal{W}$:\\
(i) The graphon operator $T_h$ is self-adjoint, bounded and continuous.\\
(ii) The graphon operator $T_h$ is diagonalisable and has countably many eigenvalues, all of which are real and can be ordered as $\mu_1 \geq \mu_2 \geq \dots \geq 0$. Moreover, there exists a collection of eigenfunctions which form an orthonormal basis of $L^2([0,1])$.\\
(iii) The maximal eigenvalue $\mu_1$ of the graphon operator $h$ is strictly positive and has an associated eigenfunction $u_1$ satisfying $u_1(x) > 0$ for all $x \in [0,1]$. Moreover, $\mu_1 = \| T_h \|$, i.e., the maximal eigenvalue equals the operator norm.
\end{proposition}

\begin{proof}
The claim is a special case of \cite[Theorem 7.3]{sauvigny} (when the compact Hermitian operators considered there are taken to be the graphon operators). See also \cite[Theorem 19.2]{deimling} and \cite[Appendix A]{disertori}.
\end{proof}


\subsection{Main theorems} 

Let $\lambda_N$ be the \emph{maximal eigenvalue} of the adjacency matrix $A_N$ of $G_N$. Write $\mathbb{P}^*_N$ to denote the law of $\lambda_N/N$.

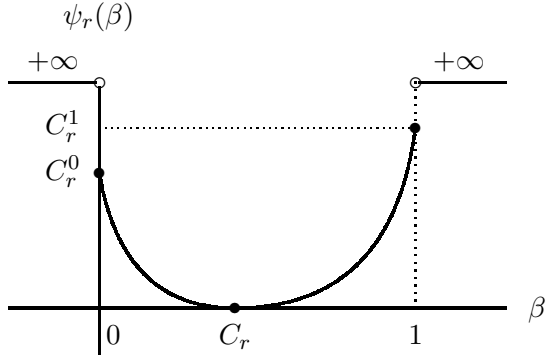
\begin{figure}[htbp]
\vspace{0.3cm}
\begin{center}
\setlength{\unitlength}{0.6cm}
\begin{picture}(10,6)(0,0)
{\thicklines
\qbezier(0,-1)(0,2)(0,4.9)
\qbezier(-2,0)(4,0)(9,0)
\qbezier(3,0)(6.5,0)(7,4)
\qbezier(3,0)(0.5,0)(0,3)
\qbezier(7.1,5)(8,5)(9,5)
\qbezier(-0.1,5)(-1,5)(-2,5)
}
\qbezier[30](7,0)(7,2.5)(7,5)
\qbezier[50](0,4)(3.5,4)(7,4)
\put(2.65,-.8){$C_r$}
\put(6.85,-.8){$1$}
\put(0.15,-.8){$0$}
\put(-1.2,3.85){$C^1_r$}
\put(-1.2,2.85){$C^0_r$}
\put(9.5,-.2){$\beta$}
\put(-.8,6.3){$\psi_r(\beta)$}
\put(7.4,5.3){$+ \infty$}
\put(-1.6,5.3){$+ \infty$}
\put(0,3){\circle*{0.25}}
\put(3,0){\circle*{0.25}}
\put(7,4){\circle*{0.25}}
\put(6.85,4.85){$\circ$}
\put(-0.15,4.85){$\circ$}
\end{picture}
\end{center}
\vspace{0.5cm}
\caption{\small Graph of $\beta \mapsto \psi_r(\beta)$.}
\label{fig-rf}
\vspace{0.2cm}
\end{figure}

\begin{theorem}{\bf [LDP for the maximal eigenvalue]}
\label{thm:main}
Subject to \eqref{Assbasic}, the sequence $(\mathbb{P}^*_N)_{N\in\mathbb{N}}$ satisfies the large deviation principle on $\mathbb{R}$ with rate $\binom{N}{2}$ and with rate function
\begin{equation}
\label{rf}
\psi_r(\beta) 
= \inf_{\substack{\widetilde{h} \in \widetilde{\mathcal{W}} \\ \|T_h\| = \beta}} J_r(\widetilde{h}) 
= \inf_{\substack{h \in \mathcal{W} \\ \|T_h\| = \beta}} I_r(h), \qquad \beta \in \mathbb{R}.
\end{equation}
\end{theorem}

\begin{proof}
Note that $\lambda_N/N = \|T_{h^{G_N}}\|$, where $h$ is any representative of $\widetilde h$ (because $\|T_{\widetilde{h}}\|=\|T_{h^\phi}\|$ for all $\phi\in\mathcal{M}$). Also note that $\widetilde{h} \mapsto \|T_{\widetilde{h}}\|$ is a bounded and continuous function on $\widetilde{\mathcal{W}}$ \cite[Exercises 6.1--6.2, Lemma 6.2]{chatterjeeln2017}. Hence the claim follows from Theorem \ref{thm:LDPinhom} via the contraction principle \cite[Chapter 3]{dH2000}.
\end{proof}

Put 
\begin{equation}
C_r = \| T_r \|.
\end{equation}
When $\beta = C_r$, the graphon $h$ that minimizes $I_r(h)$ such that $\|T_h\| = C_r$ is the reference graphon $h=r$ almost everywhere, for which $I_r(r) = 0$ and no large deviation occurs. When $\beta > C_r$, we are looking for graphons $h$ with a larger operator norm. The large deviation cannot go above 1, which is represented by the constant graphon $h \equiv 1$, for which $I_r(1) = C^1_r$. Similarly, when $\beta < C_r$, we are looking for graphons $h$ with a smaller operator norm. The large deviation cannot go below 0, which is represented by the constant graphon $h \equiv 0$, for which $I_r(0) = C^0_r$ (see Fig.~\ref{fig-rf}).

\begin{theorem}{\bf [Properties of the rate function]}
\label{thm:rfprop}
Subject to \eqref{Assbasic}:\\
(i) $\psi_r$ is continuous and unimodal on $[0,1]$, with a unique zero at $C_r$.\\
(ii) $\psi_r$ is strictly decreasing on $[0,C_r]$ and strictly increasing on $[C_r,1]$.\\
(iii) For every $\beta \in [0,1]$, the set of minimisers of the variational formula for $\psi_r(\beta)$ in \eqref{rf} is non-empty 
and compact in $\widetilde{\mathcal{W}}$.
\end{theorem}

If the reference graphon $r$ is of rank 1, i.e.,
\begin{equation}
\label{rank}
r(x,y) = \nu(x)\, \nu(y), \quad (x,y) \in [0,1]^2, 
\end{equation}
for some $\nu\colon\,[0,1] \to [0,1]$ that is bounded away from $0$ and $1$, then we are able to say more. 
Define 
\begin{equation}
\label{mkdef}
m_k = \int_{[0,1]} \nu^k, \qquad k \in \mathbb{N}.
\end{equation}
Note that $C_r = m_2$. Abbreviate $B_r= \int_{[0,1]^2} r^3(1-r)$, and note that $B_r=m_3^2-m_4^2$. Further abbreviate
\begin{equation}
N^1_r = \int_{[0,1]^2} \frac{1-r}{r}, \qquad 
N^0_r = \int_{[0,1]^2} \frac{r}{1-r}.
\end{equation}
Recall that $\mathcal{M}$ is the set of Lebesgue measure-preserving bijective maps $\phi\colon\,[0,1] \to [0,1]$. 

\begin{theorem}{\bf [Scaling of the rate function]}
\label{thm:rfscaling}
Let $\psi_r$ be the rate function in \eqref{rf}.\\
(i) Subject to \eqref{Assbasic} and \eqref{rank},  
\begin{equation}
\label{scal1}
\psi_r(\beta) = [1+o(1)]\,K_r\,(\beta- C_r)^2, \qquad \beta \to C_r,
\end{equation}
with 
\begin{equation}
\label{Kr,moments}
K_r = \frac{C_r^2}{2B_r} = \frac{m_2^2}{2(m_3^2-m_4^2)}.
\end{equation}
(ii) Subject to \eqref{Assbasic},
\begin{equation}
\label{scal2}
C^1_r - \psi_r(\beta) = (1-\beta) \left[\log\left(\frac{N^1_r}{1-\beta}\right)+1+\mathrm{o}(1)\right],  \qquad \beta \uparrow 1.
\end{equation}
(iii) Subject to \eqref{Assbasic},
\begin{equation}
\label{scal3}
C^0_r - \psi_r(\beta) = \beta \left[\log\left(\frac{N^0_r}{\beta}\right)+1+\mathrm{o}(1)\right], \qquad \beta \downarrow 0.
\end{equation}
\end{theorem}

\begin{theorem}{\bf [Scaling of the minimisers]}
\label{thm:minscaling}
Let $h_\beta \in \mathcal{W}$ be  any minimiser of the second infimum in $\eqref{rf}$.\\
(i) Subject to \eqref{Assbasic} and \eqref{rank},
\begin{equation}
\label{scal4}
\lim_{\beta \to C_r} (\beta-C_r)^{-1} \|h_\beta -r - (\beta-C_r)\Delta\|_2 = 0,
\end{equation}
with
\begin{equation}
\label{scal5}
\Delta(x,y) = \frac{C_r}{B_r}\,r(x,y)^2[1-r(x,y)],
\qquad (x,y) \in [0,1]^2.
\end{equation}
(ii) Subject to \eqref{Assbasic},
\begin{equation}
\label{scal6}
\lim_{\beta \uparrow 1} (1-\beta)^{-1} \|1-h_\beta - (1-\beta)\Delta\|_2= 0,
\end{equation} 
with
\begin{equation}
\label{scal7}
\Delta(x,y) = \frac{1}{N^1_r}\,\frac{1-r(x,y)}{r(x,y)}, 
\qquad (x,y) \in [0,1]^2.
\end{equation}
(iii) Subject to \eqref{Assbasic},
\begin{equation}
\label{scal8}
\lim_{\beta \downarrow 0} \beta^{-1} \|h_\beta - \beta\Delta\|_2 = 0,
\end{equation} 
with
\begin{equation}
\label{scal9}
\Delta(x,y) = \frac{1}{N^0_r}\,\frac{r(x,y)}{1-r(x,y)}, 
\qquad (x,y) \in [0,1]^2. 
\end{equation}
\end{theorem}


\subsection{Discussion and outline}

$\mbox{}$

\medskip\noindent
{\bf 1.}
Theorem~\ref{thm:rfprop} confirms the picture of $\psi_r$ drawn in Fig.~\ref{fig-rf}. It remains open whether or not $\psi_r$ is convex. We do \emph{not} expect $\psi_r$ to be analytic, because bifurcations may occur in the set of minimisers of $\psi_r$ as $\beta$ is varied. 

\medskip\noindent
{\bf 2.}
Theorem~\ref{thm:rfscaling} identifies the scaling of $\psi_r$ near its zero and near its end points, provided $r$ is of rank $1$. Theorem~\ref{thm:minscaling} identifies the corresponding scaling of the minimiser $h_\beta$ of $\psi_r$. Interestingly, the scaling corrections are \emph{not} rank $1$. It remains open to determine what happens near $C_r$ when $r$ is not of rank $1$ (see the Appendix).

\medskip\noindent
{\bf 3.}
The inverse curvature $1/K_r$ equals the variance in the central limit theorem derived in \cite{chakrabarty2019ext}. This is in line with the standard folklore of large deviation theory.

\medskip\noindent
{\bf 4.}
It would be interesting to investigate to what extent the condition on the reference graphon in \eqref{Assbasic} can be weakened to some form of integrability condition. Especially for the upper bound in the LDP this is delicate, because the proof in \cite{souvik2020} is based on \emph{block-graphon approximation} (see \cite{Markering2020}).

\medskip\noindent
{\bf Outline.}
The proof of Theorems~\ref{thm:rfprop}--\ref{thm:minscaling} is given in Section~\ref{sec:rf} and relies on the variational formula in \eqref{rf}. Since the maximal eigenvalue is invariant under relabeling of the vertices, we can work directly with $I_r$ in \eqref{Irhdef} without worrying about the equivalence classes. In Section~\ref{sec:exp} we derive an expansion for the operator norm of a graphon around \emph{any} graphon of rank 1. This expansion will be needed in Section~\ref{sec:rf}.


\section{Expansion around graphons of rank $1$}
\label{sec:exp}

In order to prepare for the proof of Theorem~\ref{thm:rfscaling}, we show how we can expand the operator norm of a graphon around \emph{any} graphon of rank 1.

\begin{lemma}{\bf [Rank 1 expansion]}
\label{lemmaexpansion}
Consider a graphon $\bar{h} \in \mathcal{W}$ of rank 1 such that $\bar{h}(x,y) = \bar{\nu}(x)\bar{\nu}(y)$, $x,y \in [0,1]$. For any $h \in \mathcal{W}$ such that $\|T_{h-\bar{h}} \| < \|T_h\|$, the operator norm $\mu=\|T_h\|$ is a solution of the equation
\begin{equation}
\label{expansion}
\mu = \sum_{n \in \N_0} \frac{1}{\mu^n} \mathcal{F}_n(h, \bar{h}),
\end{equation}
where
\begin{equation}
\label{Fn}
\mathcal{F}_n(h,\bar{h}) = \int_{[0,1]^2} \d x \, \d y \, \bar{\nu}(x) (h-\bar{h})^n(x,y) \bar{\nu}(y).
\end{equation}
\end{lemma}

\begin{proof}
By Proposition~\ref{prop:basics}, we have
\begin{equation}
T_hu=\mu u,
\end{equation}
where $\mu$ equals both the norm and the maximal eigenvalue of $T_h$, and $u$ is the eigenfunction of $h$ corresponding to $u$. Put $g = h - \bar{h}$ and we have  $(\mu - T_g) u = T_{\bar{h}}u$. This gives
\begin{equation}
u = (\mu - T_g)^{-1} \bar{\nu} \langle \bar{\nu},u\rangle.
\end{equation}
where we use that $\mu - T_g$ is invertible because $\| T_g \| = \|T_{h-\bar{h}} \| < \| T_h \|$. Hence, taking the inner product of $u$ with $\bar{\nu}$ and observing that $\la \bar{\nu}, u\ra \neq 0$, we get
\begin{equation}
\langle \bar{\nu}, u \rangle = \langle\bar{\nu}, u\rangle \langle\bar{\nu}, (\mu - T_g)^{-1} \bar{\nu} \rangle 
\end{equation}
which gives
\begin{equation}
\mu = \langle\bar{\nu},(1 - T_g/\mu)^{-1} \bar{\nu} \rangle.
\end{equation}

We can expand the above to get
\begin{equation}
\label{eqlambda}
\begin{split}
\mu &= \bigg\langle\bar{\nu}, \sum_{n \in \mathbb{N}_0} \left(\frac{T_g}{\mu}\right)^{n} \bar{\nu} \bigg\rangle \\
& = \sum_{n \in \N_0} \frac{1}{\mu^n} \int_{[0,1]^{n+1}} \d x_0 \, \d x_1 \cdots \d x_{n}\,
\bar{\nu}(x_0)g(x_0, x_1) \times \cdots \times g(x_{n-1}, x_n) \bar{\nu}(x_n)\\
& =  \sum_{n \in \N_0} \frac{1}{\mu^n} \mathcal{F}_n(h, \bar{h}).
\end{split}
\end{equation}
Since $\|T_h\| = \mu$, we get \eqref{expansion}. 
\end{proof}

Subject to \eqref{rank}, it follows from Lemma~\ref{lemmaexpansion} with $h=\bar{h}=r$ that
\begin{equation}
\label{Crrank}
C_r = \| T_r \| = m_2,
\end{equation}
because only the term with $n=0$ survives in the expansion.  

\begin{remark}{\bf [Higher rank]} 
{\rm The expansion around reference graphons of rank 1 can be extended to finite rank. We provide the details in the Appendix. In this paper we focus on rank 1, for which Lemma~\ref{lemmaexpansion} allows us to analyse the behaviour of $\psi_r(\beta)$ near the values $\beta=C_r$, $\beta=1$ and $\beta=0$. This corresponds to an expansion around the graphons $h=\nu\times\nu$, $h \equiv 1$ and $h \equiv 0$, which are all of rank 1.} \hfill$\blacksquare$
\end{remark}


\section{Proofs of main theorems}
\label{sec:rf}

Theorem~\ref{thm:rfprop} is proved in Section~\ref{sec:conuni}, Theorems~\ref{thm:rfscaling}--\ref{thm:minscaling} are proved in Sections~\ref{pertCr}--\ref{pert0}. 


\subsection{Continuity, unimodality and unique minimisers}
\label{sec:conuni} 

\begin{proof}
We follow \cite[Chapter 6]{chatterjeeln2017}. Even though this monograph deals with constant reference graphons only, most arguments carry over to $r$ satisfying \eqref{Assbasic}.
   
\medskip\noindent   
(i), (iii) Define
\begin{equation}
\label{phi+-}
\psi_r^+(\beta) = \inf_{\substack{h \in \mathcal{W} \\ \|T_h\| \geq \beta}} I_r(h), 
\quad
\psi_r^-(\beta) = \inf_{\substack{h \in \mathcal{W} \\ \|T_h\| \leq \beta}} I_r(h), 
\qquad \beta \in \mathbb{R}.
\end{equation}
Because $h \mapsto \|T_h\|$ is a nice graph parameter, in the sense of \cite[Definition 6.1]{chatterjeeln2017}, it follows that $\beta \mapsto \psi_r^+(\beta)$ is non-decreasing and continuous, while $\beta \mapsto \psi_r^-(\beta)$ is non-increasing and continuous \cite[Proposition 6.1]{chatterjeeln2017}. (The proof requires the fact that $\| f_n - f \|_2 \to 0$ implies $I_r(f_n) \to I_r(f)$ and that $I_r(f)$ is lower semi-continuous on $\mathcal{W}$.)

The variational formulas in \eqref{phi+-} achieve minimisers. In fact, the sets of minimiser are non-empty compact subsets of $\widetilde{\mathcal{W}}$ \cite[Theorem 6.2]{chatterjeeln2017}. In addition, all minimisers $h$ of $\phi^+_r(h)$ satisfy $h \geq r$ almost everywhere, while all minimisers $h$ of $\phi^-_r$ satisfy $h \leq r$ almost everywhere \cite[Lemma 6.3]{chatterjeeln2017}. Moreover, because 
\begin{equation}
\begin{aligned}
&h_1 \geq h_2 \geq r \quad \Longrightarrow \quad \|T_{h_1}\| \geq \|T_{h_2}\|, \,\,I_r(h_1) \geq I_r(h_2),\\  
&h_1 \leq h_2 \leq r \quad \Longrightarrow \quad \|T_{h_1}\| \leq \|T_{h_2}\|, \,\,I_r(h_1) \leq I_r(h_2),  
\end{aligned}
\end{equation}
(use that $a \mapsto \mathcal{R}(a \mid b)$ is unimodal on $[0,1]$ with unique zero at $b$), it follows that both variational formulas achieve minimisers with norm equal to $\beta$, and so
\begin{equation}
\psi_r(\beta) = \left\{\begin{array}{ll}
\psi_r^+(\beta), &\beta \geq C_r,\\
\psi_r^-(\beta), &\beta \leq C_r. 
\end{array}
\right.
\end{equation}
Hence, $\psi_r$ is continuous and unimodal on $[0,1]$. Since $I_r(h) = 0$ if and only if $h=r$ almost everywhere, it is immediate that $C_r$ is the unique zero of $\psi_r$.     

\medskip\noindent
(ii)
The proof is by contradiction. Suppose that $\beta\mapsto\psi_r^+(\beta)$ is not strictly increasing on $[C_r,1]$. Then there exist $\beta_1,\beta_2 \in [C_r,1]$ with $\beta_1<\beta_2$ such that $\psi_r^+$ is constant on $[\beta_1,\beta_2]$. Consequently, there exist minimisers $h_{\beta_1}^{\phi_1}, h_{\beta_2}^{\phi_2}$ with $\phi_1, \phi_2 \in \mathcal{M}$ satisfying $r \leq h_{\beta_1}^{\phi_1} \leq h_{\beta_2}^{\phi_2}$ such that
\begin{equation}
I_r(h_{\beta_1}^{\phi_1}) = I_r(h_{\beta_2}^{\phi_2}), \qquad 
\|T_{h_{\beta_1}^{\phi_1}}\| = \beta_1 <  \beta_2 = \|T_{h_{\beta_2}^{\phi_2}}\|.
\end{equation} 
However, since $a \mapsto \mathcal{R}(a \mid b)$ is strictly increasing on $[b,1]$ (recall \eqref{Irhdef}), it follows that $h_{\beta_1}^{\phi_1} = h_{\beta_2}^{\phi_2}$ almost everywhere. This in turn implies that $\|T_{h_{\beta_1}^{\phi_1}}\| = \|T_{h_{\beta_2}^{\phi_2}}\|$, which is a contradiction. A similar argument shows that  $\beta\mapsto\psi_r^-(\beta)$ cannot have a flat piece on $[0,C_r]$.
\end{proof}


\subsection{Perturbation around the minimum}
\label{pertCr}

Note that when $\beta = C_r$, the infimum in \eqref{rf} is attained at $h=r$ and $\psi_r(C_r) = 0$. Take $\beta = C_r +\epsilon$ with $\epsilon > 0$ small, and assume that the infimum is attained by a graphon of the form $h = r + \Delta_{\epsilon}$, where $\Delta_{\epsilon} \colon\,[0,1]^2 \to \R$ represents a perturbation of the graphon $r$. Note that $r + \Delta_{\epsilon} \in \mathcal{W}$, and so we are dealing with a perturbation $\Delta_{\epsilon}$ that is symmetric and bounded. We compare
\begin{equation}\label{eq:diff0}
\psi_r(C_r + \epsilon) = \inf_{\substack{\Delta_{\epsilon} \colon\,[0,1]^2 \to \R \\ r + \Delta_{\epsilon} \in \mathcal{W} \\ \|T_{r+\Delta_{\epsilon}}\| = C_r + \epsilon}} I_r(r+\Delta_{\epsilon}) 
\end{equation}  
with $\psi_r(C_r)=0$ by computing the difference
\begin{equation}
\label{diff}
\delta_r(\epsilon) = \psi_r(C_r + \epsilon) - \psi_r(C_r) = \psi_r(C_r + \epsilon)
\end{equation}
and studying its behaviour as $\epsilon \to 0$. Since $r(x,y) = \nu(x) \nu(y)$, $x,y \in [0,1]$, we can use Lemma~\ref{lemmaexpansion} to control the norm of $T_h = T_{r + \Delta_{\epsilon}}$. Pick $\bar{h}=r$ and $h=r + \Delta_{\epsilon}$ in \eqref{expansion} such that $\|\Delta_{\epsilon}\|_2\to 0$ as $\epsilon \to 0$.  Note that $\| T_{\Delta_{\epsilon}} \|\le \|\Delta_{\eps}\|_2 < C_r$ for $\epsilon$ small enough. Hence, writing out the expansion for the norm, we get 
\begin{equation}
\|T_{r+\Delta_{\epsilon}}\| = C_r + \sum_{n \in \N} \frac{1}{\|T_{r+\Delta_{\epsilon}}\|^n}\, 
\mathcal{F}_n(r+\Delta_{\epsilon}, r).
\end{equation}

Since $\|T_{r+\Delta_{\epsilon}}\|= C_r+\epsilon$, we have
\begin{equation}
C_r+\epsilon= C_r+ \frac{ \la \nu, \Delta_{\epsilon} \nu\ra}{C_r+\epsilon} 
+ \sum_{n\in\N \setminus \{1\}} \frac{1}{(C_r+\epsilon)^n} \la \nu, \Delta_{\epsilon}^n\nu\ra
\end{equation}
with $\la \nu, \Delta_{\epsilon} \nu \ra= \int_{[0,1]^2} r\Delta_{\epsilon}$. So
\begin{equation}
\epsilon( C_r+\epsilon)= \int_{[0,1]^2} r\Delta_{\epsilon} +\sum_{n\in\N\setminus\{1\}} 
\frac{1}{(C_r+\epsilon)^{n-1}}\la \nu, \Delta_{\epsilon}^n \nu\ra.
\end{equation}
Since $\nu$ is bounded, using the generalized H\"older's inequality \cite[Theorem 3.1]{lubetzkyzhao2015} we get 
\begin{equation} 
|\la \nu, \Delta_{\epsilon}^n \nu\ra| \le \|\Delta_{\epsilon}\|_2^{n}.
\end{equation}
Since $\|\Delta_{\eps}\|_2\to 0$ as $\eps \to 0$, we can choose $\epsilon$ small enough such that $\|\Delta_{\eps}\|_2 < \tfrac12(C_r+\epsilon)$, which gives 
\begin{equation}
\sum_{n\in\N\setminus\{1\}} \frac{1}{(C_r+\epsilon)^{n-1}}\la \nu, \Delta_{\epsilon}^n \nu\ra
= \mathrm{O}\left( \|\Delta_{\epsilon}\|_2^2\right).
\end{equation}
The constraint $\| r + \Delta_{\epsilon} \| = C_r + \epsilon$ therefore reads
\begin{equation}
\label{constrCr}
\int_{[0,1]^2} r\Delta_{\eps}= \eps C_r+ \eps^2+ \mathrm{O}\left( \|\Delta_{\eps}\|_2^2\right).
\end{equation}
Observe that if $\Delta_{\eps}= \eps\Delta$ for some function $\Delta\in L^2([0,1]^2)$, then
\begin{equation}
\int_{[0,1]^2} r\Delta = C_r\left[1+\mathrm{o}(1)\right].
\end{equation}


\subsubsection{Small perturbation on a given region} 

In what follows we use the standard notation $o(\cdot)$, $O(\cdot)$, $\asymp$ to describe the asymptotic behaviour in the limit as $\epsilon \to 0$. We start by considering different types of small perturbations in a given region and computing their total cost.


We are interested in the finding the asymptotic behaviour of \eqref{eq:diff0}. In the next lemma we show that it is enough to consider $\Delta_\eps$ of the form $\eps \Delta$ for some $\Delta\in L^2([0,1]^2)$, because these perturbations contribute to the minimum cost.

\begin{lemma}{\bf [Order of minimal cost]}
\label{lem:costperturbation0}
Let $\Delta_\eps:[0,1]^2\to \mathbb R$ be such that $r+\Delta_\eps\in \mathcal W$ and $\|T_{r+\Delta_{\eps}}\|=C_r+\eps$. Then 
\begin{equation}
\label{eq:inequal}
I_r(r+ \Delta_\eps) \ge 2\eps^2.
\end{equation}
Moreover, if $\Delta_{\eps}= \eps\Delta$, then
\begin{equation}
\label{eq:costasymp}
I_r(r+\eps\Delta) = [1+o(1)]\,2\eps^2 \int_{[0,1]^2} \frac{\Delta^2}{4r(1-r)}, 
\qquad \eps\to 0.
\end{equation}
\end{lemma}

\begin{proof}
Fix $b\in [0,1]$ and abbreviate (recall \eqref{Rdef})
\begin{equation}
\label{eq:chidef}
\chi(a) = \mathcal{R}(a \mid b) = a \log \frac{a}{b} + (1-a) \log \frac{1-a}{1-b}, \qquad a \in [0,1].
\end{equation}
Note that 
\begin{equation}
\chi(b) = \chi^\prime(b)=0, \quad \chi^{\prime \prime}(a) \geq 4, \quad a \in  [0,1].
\end{equation}
Consequently, 
\begin{equation} 
\chi(a) \ge 2 (a- b)^2, \qquad  a\in [0,1],
\end{equation}
and hence
\begin{equation}
\label{eq:lbd1}
I_r( r+\Delta_\eps) \geq 2 \int_{[0,1]^2} \Delta_\eps^2 = 2\| \Delta_\eps\|_2^2.
\end{equation}
Next observe that
\begin{equation}
C_r+\eps = \|T_{r+\Delta_{\eps}}\|= \|T_r+T_{\Delta_\eps}\|
\le \|T_r\|+ \|T_{\Delta_\eps}\|\le C_r+\|\Delta_\eps\|_2,
\end{equation}
which gives $\|\Delta_{\eps}\|_2\ge \eps$. Inserting this lower bound into \eqref{eq:lbd1}, we get \eqref{eq:inequal}. To get  \eqref{eq:costasymp}, we need a higher-order expansion of $\chi$, namely, $\chi(x) = \tfrac12 \chi''(b)(x-b)^2 + \mathrm{O}((x-b)^3)$, $x \to b$. Since $r$ is bounded away from $0$ and $1$, and the constraint $r+\Delta_\eps\in \mathcal W$ implies that $\Delta_{\eps}(x,y) \in [-1, 1]$, we see that the third-order term is smaller than the second-order term when $\Delta_\eps= \eps \Delta$. Hence \eqref{eq:costasymp} follows. 
\end{proof}

\begin{lemma}{\bf [Cost of small perturbations]}
\label{lem:costperturbation}
Let $B \subseteq [0,1]^2$ be a measurable region with area $|B|$. Suppose that $\Delta_{\epsilon} = \epsilon^{\alpha} \Delta$ on $B$, with $\epsilon > 0$, $\alpha >0$ and $\Delta\colon [0,1]^2 \to \R$. Then the contribution of $B$ to the cost $I_r(h)$ is 
\begin{equation}
\int_B \d x \, \d y  \, \mathcal{R}(h(x,y) \mid r(x,y)) 
= [1+\mathrm{o}(1)] \,\epsilon^{2\alpha} \int_{B} \frac{\Delta^2}{2r(1-r)}, \qquad \epsilon \to 0.
\end{equation}
If the integral diverges, then the contribution decays slower than $\epsilon^{2\alpha}$. 
\end{lemma}

\begin{proof}
The proof is similar to that of Lemma~\ref{lem:costperturbation0}.
\end{proof}


\subsubsection{Approximation by block graphons}

We next introduce block graphons, which will be useful for our perturbation analysis. It follows from Lemma~\ref{lem:costperturbation0} that optimal perturbations with $\Delta_\eps$ must satisfy $\|\Delta_{\eps}\|_2\asymp \eps$, and hence it is desirable to have $\Delta_\eps= \eps \Delta$. We argue through block graphon approximations that this is indeed the case.

\begin{definition}{\bf [Block graphons]}
\label{def:blockgraphons}
Let $\mathcal{W}_N\subset\mathcal W$ be the space of graphons with $N$ blocks having a constant value on each of the blocks, i.e., $f\in \mathcal W_N$ is of the form
\begin{equation}
f(x,y)= \begin{cases}
f_{i,j}, & \, \, (x,y)\in B_i \times B_j,\\
0, & \,\, \text{otherwise},
\end{cases}
\end{equation}
where $B_i=[ \frac{i-1}{N}, \frac{i}{N})$, $1\le i\le N-1$ and $B_N=[\frac{N-1}{N}, 1]$ and $f_{i,j}\in [0,1]$. Write $B_{i,j}= B_i\times B_j$. With each $f \in \mathcal{W}$ associate the block graphon $f_N \in \mathcal{W}_N$ given by
\begin{equation}
f_N(x,y) =  N^2 \int_{B_{i,j}}  \d x' \, \d y' \, f(x', y') = \bar{f}_{N, ij}, \quad (x,y) \in B_{i,j}.
\end{equation}
\end{definition}

Observe that if $f_N$ is the block graphon associated with a graphon $f$, then
\begin{equation}
\|T_{f_N}- T_f\|= \|T_{f_N-f}\|\le \|f_N-f\|_2.
\end{equation}
We know from \cite[Proposition 2.6]{chatterjeeln2017} that $\|f_N-f\|_2\to 0$, and hence $\lim_{N\to\infty} \|T_{f_N}\| = \| T_f \|$ for any $f \in \mathcal{W}$ and its associated sequence of block graphons $(f_N)_{N \in \mathbb{N}}$. The following lemma shows that the cost function associated with the graphons $r$ and $f$ is well approximated by the cost function associated with the block graphons $r_N$ and $f_N$.

\begin{lemma}{\bf [Convergence of the cost function]}
\label{lem:convergencecost}
$\lim_{N\to\infty} I_{r_N}(f_N) = I_{r} (f)$ for any $f \in \mathcal{W}$.
\end{lemma}

\begin{proof}
Since $f \in L^2([0,1]^2)$, $f_N$ is bounded. The assumption in \eqref{Assbasic} implies that $\eta \leq r_N \leq 1-\eta$ for all $N \in \mathbb{N}$. We know from \cite[Lemma 2.3]{souvik2020} that there exists a constant $c > 0$ independent of $f$ such that
\begin{equation}
| I_{r_N}(f) - I_r(f) | \leq c\,  \| r_N - r \|_1 \leq c \, \| r_N - r \|_2.
\end{equation}
Hence
\begin{equation}
\begin{split}
| I_{r_N}(f_N) - I_r(f) | & \leq | I_{r_N}(f_N) - I_r(f_N) | + | I_{r}(f_N) - I_r(f) | \\
& \leq c \, \| r_N - r \|_2 + | I_{r}(f_N) - I_r(f) |.
\end{split}
\end{equation}
But we know from \cite[Proposition 2.6]{chatterjeeln2017} that $\lim_{N\to\infty} \| r_N - r \|_2 = 0$ and $\lim_{N\to\infty} \| f_N - f \|_2 = 0$, while we know from \cite[Lemma 3.4]{Markering2020} that $I_r$ is continuous in the $L^2$-topology on $\mathcal{W}$.
\end{proof}


\subsubsection{Block graphon perturbations}

In what follows we fix $N\in\N$, analyse different types of perturbation and identify which one is optimal. For each $N\in\N$, we associate with the perturbed graphon $h = r + \Delta_{\epsilon}$ the block graphon $h_N \in \mathcal{W}_N$ given by
\begin{equation}
\overline{h}_{N, ij}(x,y)  = \overline{r}_{N, ij}(x,y)  + \overline{\Delta_{\epsilon}}_{N, ij}(x,y) , \quad (x,y) \in B_{i,j},
\end{equation}
with 
\begin{equation}
\overline{r}_{N,ij} = N^2 \int_{B_{i,j}}  \d x' \, \d y' \, r(x', y'), \qquad
\overline{\Delta_{\epsilon}}_{N, ij}= N^2 \int_{B_{i,j}}  \d x' \, \d y' \, \Delta_{\epsilon}(x', y') .
\end{equation}
Observe that optimal perturbations must have $\|\Delta_\eps\|_2= \mathrm{O}(\eps)$, and hence the constraint in \eqref{constrCr} becomes
\begin{equation}
\label{blockconstraint}
\sum_{i,j=1}^N \int_{B_{i,j}} r \Delta_{\epsilon} 
= \sum_{i,j = 1}^N \frac{1}{N^2}\,\overline{r \Delta_{\epsilon}}_{N, ij} = [1+\mathrm{o}(1)]\,C_r \epsilon, 
\qquad \epsilon \to 0.
\end{equation}
The block constraint in \eqref{blockconstraint} implies that the sum over each block must be of order $\epsilon$. We therefore must have that
\begin{equation}
\overline{r \Delta_{\epsilon}}_{N, ij} = \mathrm{O}(\epsilon), \qquad \epsilon \to 0 \qquad \forall\, (i,j),
\end{equation}
which means that 
\begin{equation}
\overline{\Delta_{\epsilon}}_{N, ij} = \mathrm{O}(\epsilon), \qquad \epsilon \to 0 \qquad \forall\, (i,j),
\end{equation}
since \eqref{Assbasic} implies that $\overline{r \Delta_{\epsilon}}_{N, ij} \asymp\overline{\Delta_{\epsilon}}_{N, ij}$. There are two possible cases:
\begin{itemize}
\item[(I)] 
All blocks contribute to the constraint with a term of order $\epsilon$ (balanced perturbation).
\item[(II)] 
Some blocks contribute to the constraint with a term of order $\epsilon$ and some with $o(\epsilon)$ (unbalanced perturbation).
\end{itemize}
\noindent
Perturbations of type (I) consist of a small perturbation on each block, i.e., $\overline{\Delta_{\epsilon}}_{N, ij} \asymp{\epsilon}$ for each block $B_{i,j}$. By Lemma \ref{lem:costperturbation}, this contributes a term of order $\epsilon^2$ to the total cost. Since all blocks have the same type of perturbation, they all contribute in the same way, and so we get $I_{r_N}(h_N) \asymp \epsilon^2$. We will see in Corollary \ref{cor:lessarea} that perturbations of type (II) are worse than perturbations of type (I). Let $1 \leq k \leq N^2-1$ be the number of blocks that contribute a term of order $\mathrm{o}(\epsilon)$ to the constraint, i.e., $\overline{\Delta_{\epsilon}}_{N, ij} = \mathrm{o}(\epsilon)$. By Lemma \ref{lem:costperturbation}, these blocks contribute order $\mathrm{o}(\epsilon^2)$ to the total cost. The remaining blocks must fall in the class of blocks of type (I), with a perturbation of order $\epsilon$ on each of them. Lemma \ref{cor:lessarea} below shows that the cost function attains its infimum when the perturbation of order $\epsilon$ is uniform on $[0,1]^2$. 


\subsubsection{Optimal perturbation}

We have shown that perturbations of type (I) lead to the minimal total cost. They consist of small perturbations of order $\epsilon$ on all blocks, and hence on $[0,1]^2$. A sequence of such perturbations $(\Delta_{\epsilon, N})_{N \in \mathbb{N}}$ converges to a perturbation $\Delta_{\epsilon}$ as $N \to \infty$. We can identify the cost of $\Delta_{\epsilon} = \epsilon \Delta$ with $\Delta\colon\,[0,1]^2 \to \R$, which we refer to as a balanced perturbation.

\begin{lemma}{\bf [Balanced perturbations]}
\label{lem:globalperturbation}
Suppose that $\Delta_{\epsilon}=\epsilon \Delta$ with $\Delta\colon\,[0,1]^2 \to \R$. Let $\mathcal{M}$ be the set of Lebesgue measure-preserving bijective maps. Then
\begin{equation}
\delta_r(\epsilon) = [1+\mathrm{o}(1)] \, K_r \epsilon^2, \qquad \epsilon \to 0,
\end{equation}
with
\begin{equation}
\label{KCr*}
K_r = \frac{1}{2}C_r^2 \inf_{\phi \in \mathcal{M}}
\frac{D_r^{\phi}}{(B_r^{\phi} )^2},
\end{equation}
where $B_r^{\phi} = \int_{[0,1]^2} r^{\phi} r^2(1-r)$ and $D_r^{\phi} = \int_{[0,1]^2} (r^{\phi})^2 r(1-r)$.
\end{lemma}

\begin{proof}
The constraint in \eqref{constrCr} becomes
\begin{equation}
\int_{[0,1]^2} r \Delta = [1+\mathrm{o}(1)]\, C_r, \qquad \epsilon \to 0,
\end{equation}
and we get
\begin{equation}
\begin{split}
\delta_r(\epsilon) &=  \inf_{\substack{\Delta\colon\,[0,1]^2 \to \R \\ r 
+ \epsilon \Delta \in \mathcal{W} \\ \int_{[0,1]^2} r \Delta= [1+\mathrm{o}(1)]\,C_r}}  I_r(r+ \epsilon \Delta) \\
&=  \inf_{\substack{\Delta\colon\,[0,1]^2 \to \R \\ r + \epsilon \Delta \in \mathcal{W} \\\int_{[0,1]^2} r \Delta 
= [1+\mathrm{o}(1)]\,C_r}}  \int_{[0,1]^2} \d x \, \d y \, \mathcal{R}\big((r+ \epsilon \Delta)(x,y) \mid r(x,y)\big).
\end{split}
\end{equation}
By Lemma \ref{lem:costperturbation} (with $\alpha=1$), we have
\begin{equation}
\label{deltarew}
\delta_r(\epsilon) = [1+\mathrm{o}(1)] \, K_r \epsilon^2, \qquad \epsilon \to 0,
\end{equation}
with
\begin{equation}
\label{varK}
K_r = \inf_{\substack{\Delta\colon\,[0,1]^2 \to \R \\ \int_{[0,1]^2} r \Delta = C_r}} 
\int_{[0,1]^2} \frac{\Delta^2}{2r(1-r)}.
\end{equation}
The prefactor $1+ \mathrm{o}(1)$ in \eqref{deltarew} arises after we scale $\Delta$ by $1+ \mathrm{o}(1)$ in order to force $\int_{[0,1]^2} r \Delta= C_r$. Note that the optimisation problem in \eqref{varK} no longer depends on $\eps$.

We can apply the method of Lagrange multipliers to solve this constrained optimisation problem. To that end we define the Lagrangian 
\begin{equation}
\mathcal{L}_{A_r}(\Delta) = \int_{[0,1]^2} \frac{\Delta^2}{2\,r(1-r)}
+ A_r \int_{[0,1]^2} r \Delta,
\end{equation}
where $A_r$ is a Langrange multiplier. Since $\int_{[0,1]^2} r = \int_{[0,1]^2} r^{\phi}$ for any Lebesgue measure-preserving bijective map $\phi \in \mathcal{M}$, we get that the minimizer (in the space of functions from $[0,1]^2\to\R$) is of the form
\begin{equation}
\label{Deltaopt}
\Delta^{\phi} = - A_r\, r^{\phi}r(1-r), \quad \phi \in \mathcal{M}.
\end{equation}
We pick $A_r$ such that the constraint is satisfied, i.e.,
\begin{equation}
-A_r B_r^{\phi} = [1+o(1)]\,C_r
\end{equation}
with
\begin{equation}
B_r^{\phi} = \int_{[0,1]^2} r^{\phi} r^2 (1-r).
\end{equation}
We get 
\begin{equation}
\label{deltaphiCr}
\Delta^{\phi} = \frac{C_r}{B_r^{\phi}}\,r^{\phi}r(1-r), \quad \phi \in \mathcal{M},
\end{equation}
and
\begin{equation}
\label{infphi}
K_r = \inf_{\phi \in \mathcal{M}} 
\int_{[0,1]^2} \frac{(\Delta^{\phi})^2}{2 r(1-r)} 
=  \frac12C_r^2 \inf_{\phi \in \mathcal{M}} \frac{D_r^{\phi}}{(B_r^{\phi})^2}
\end{equation}
with 
\begin{equation}
D_r^{\phi} = \int_{[0,1]^2} (r^{\phi})^2 r(1-r).
\end{equation}
\end{proof}

We next show that the infimum in \eqref{infphi} is uniquely attained when $\phi$ is the identity. For this we will show that $D_r^{\phi}/(B_r^{\phi})^2 \geq 1/B_r$ with equality if and only if $\phi=\mathrm{Id}$. Indeed, write
\begin{equation}
\begin{aligned}
&B_r D_r^{\phi} - (B_r^{\phi})^2\\ 
&= \int_{[0,1]^2} \d x\,\d y\,r(x,y)[1-r(x,y)] \int_{[0,1]^2} \d\bar x\,\d\bar y\,r(\bar x,\bar y)[1-r(\bar x,\bar y)]\\
&\qquad \times \Big\{r(x,y)^2\,r^\phi(\bar x,\bar y)^2 - r(x,y)r^\phi(x,y)\,r(\bar x,\bar y)r^\phi(\bar x,\bar y)\Big\}\\
&= \int_{[0,1]^2} \d x\,\d y\,r(x,y)[1-r(x,y)] \int_{[0,1]^2} \d\bar x\,\d\bar y\,r(\bar x,\bar y)[1-r(\bar x,\bar y)]\\
&\qquad \times \tfrac12\Big\{r(x,y)^2\,r^\phi(\bar x,\bar y)^2 + r^\phi(x,y)^2\,r(\bar x,\bar y)^2 
- 2r(x,y)r^\phi(x,y)\,r(\bar x,\bar y)r^\phi(\bar x,\bar y)\Big\}\\
&= \int_{[0,1]^2} \d x\,\d y\,r(x,y)[1-r(x,y)] \int_{[0,1]^2} \d\bar x\,\d\bar y\,r(\bar x,\bar y)[1-r(\bar x,\bar y)]\\
&\qquad \times \tfrac12\Big(r(x,y)r^\phi(\bar x,\bar y)-r^\phi(x,y)r(\bar x,\bar y)\Big)^2,
\end{aligned}
\end{equation}
where the second equality uses the symmetry between the integrals. Hence $B_r D_r^{\phi} - (B_r^{\phi})^2 \geq 0$, with equality if and only if $r(x,y)/r^\phi(x,y) = C$ for almost every $x,y \in [0,1]^2$. Clearly, this can hold only for $C=1$, which amounts to $\phi = \mathrm{Id}$.  

We conclude that the infimum in \eqref{infphi} equals $1/B_r$, and so we find that
\begin{equation}
K_r = \frac{C_r^2}{2 B_r}.
\end{equation}
Finally, note that $C_r=m_2$ by \eqref{Crrank}, and that $B_r = m_3^2 - m_4^2$ by \eqref{mkdef}. This settles the expression for $K_r$ in \eqref{Kr,moments}. 

\begin{corollary}{\bf [Unbalanced perturbations]}
\label{cor:lessarea}
Perturbations of order $\epsilon$ that are not balanced, i.e., that do not cover the entire unit square $[0,1]^2$, are worse than the balanced perturbation in Lemma \ref{lem:globalperturbation}.
\end{corollary}

\begin{proof}
The argument of the variational formula can be reduced to an integral that considers only those regions that contribute order $\epsilon^2$, which constitute a subset $A \subset [0,1]^2$. Applying the method of Lagrange multipliers as in Lemma \ref{lem:globalperturbation}, we obtain that the solution is given by 
\begin{equation}
\delta_r(\epsilon) = [1+\mathrm{o}(1)]\,K'_r \epsilon^2, \qquad \epsilon \to 0,
\end{equation} 
with $K'_r > K_r$. The strict inequality comes from the fact that the optimal balanced perturbation $\Delta^\mathrm{Id}$ found in \eqref{Deltaopt} is non-zero everywhere.   
\end{proof}

In conclusion, we have shown that a balanced perturbation is optimal and we have identified in \eqref{deltaphiCr} the form of the optimal balanced perturbation. Lemma \ref{lem:globalperturbation} settles the claim in Theorem~\ref{thm:rfscaling}(i), while \eqref{deltaphiCr} settles the claim in Theorem~\ref{thm:minscaling}(i). 


\subsection{Perturbation near the right end}
\label{sec:pert1}

Take $\beta = 1 - \epsilon$ and consider a graphon of the form $h = 1 - \Delta_{\epsilon}$, where $\Delta_{\epsilon} \colon\, [0,1]^2 \to [0,\infty)$ represents a symmetric and bounded perturbation of the constant graphon $h \equiv 1$. We compare 
\begin{equation}
\psi_r(1 - \epsilon) = \inf_{\substack{\Delta_{\epsilon} \colon\, [0,1]^2 \to [0,\infty) \\ 1- \Delta_{\epsilon} \in \mathcal{W} \\ \|T_{1 - \Delta_{\epsilon}} \|= 1- \epsilon}} 
I_r(1 - \Delta_{\epsilon})
\end{equation}
with 
\begin{equation}
C^1_r = I_r(1)
\end{equation}
by computing the difference
\begin{equation}
\delta_r(\epsilon) =  \psi_r(1) - \psi_r(1 - \epsilon)
\end{equation}
and studying its behaviour as $\epsilon \downarrow 0$. Since $I_r(1)$ is a constant, we can write
\begin{equation} 
\label{variationbetalarge}
\delta_r(\epsilon) = \sup_{\substack{\Delta_{\epsilon} \colon\, [0,1]^2 \to [0,\infty) \\  1- \Delta_{\epsilon} \in \mathcal{W} \\ \|T_{1 - \Delta_{\epsilon}} \|= 1- \epsilon}} 
 [I_r(1)-I_r(1- \Delta_{\epsilon})].
\end{equation}
We again use the expansion in Lemma~\ref{lemmaexpansion}. Pick $\bar{h}=1$ and $h=1- \Delta_{\epsilon}$ in \eqref{expansion}, to get 
\begin{equation}
\|T_{1-\Delta_{\epsilon}}\| = 1 + \sum_{n \in \N} \frac{1}{\|T_{1-\Delta_{\epsilon}}\|^n}\, 
\mathcal{F}_n(1-\Delta_{\epsilon}, 1).
\end{equation}
Since $\|T_{1-\Delta_{\eps}}\|=1-\eps$, this gives
\begin{equation}
1-\eps= 1+ \frac{\la 1, (-\Delta_\eps) 1 \ra}{1-\eps}+ \frac{\la 1, (-\Delta_\eps)^2 1\ra}{(1-\eps)^2}
+\sum_{n\in\N\setminus\{1,2\}} \frac{\la 1, (-\Delta_{\eps})^{n}1\ra}{(1-\eps)^n}.
\end{equation}
For $\epsilon \downarrow 0$ we have $\|\Delta_{\epsilon}\|_2 \downarrow 0$ and $|\la 1, (-\Delta_{\eps})^n 1\ra | =\mathrm{O}(\|\Delta_{\epsilon}\|^n_2)$. Therefore
\begin{equation}
\label{eq:exp}
\eps(1-\eps)= \int_{[0,1]^2}\Delta_{\eps}- \frac{\la 1, \Delta_\eps^2 1\ra}{(1-\eps)} 
+ \mathrm{O}\left( \|\Delta_\eps\|_2^3\right).
\end{equation}
The restriction $1-\Delta_\eps\in \mathcal W$ implies that $\Delta_{\eps}\in [0,1]$. Hence $\|\Delta_{\eps}\|_2^2 \leq \|\Delta_\eps\|_1$. Moreover,
\begin{equation}
1-\eps= \|T_{1-\Delta_{\eps}}\|\le \|1-\Delta_{\eps}\|_2\le \sqrt{ \|1-\Delta_\eps\|_1}.
\end{equation}
Since $\|1-\Delta_\eps\|_1= 1- \|\Delta_\eps\|_1$, we have
\begin{equation}
\label{eq:deltaupper}
 \|\Delta_{\eps}\|_1\le 1- (1- \eps)^2 = \eps(2-\eps).
 \end{equation}
Since $\|\Delta_\eps\|_2^3 =\mathrm{O}(\eps^{3/2})$, \eqref{eq:exp} reads
\begin{equation}
\frac{1}{1-\eps} \int_{[0,1]^3} \d x\,\d y \,\d z \,\Delta_\eps(x,y)[1-\Delta_\eps(y,z)]
- \frac{\eps}{1-\eps}\,\|\Delta_\eps\|_1 = \eps(1-\eps)+ \mathrm{O}(\eps^{3/2}),
\end{equation}
which, because $\|\Delta_\eps\|_1= \mathrm{O}(\eps)$, further reduces to
\begin{equation}
\int_{[0,1]^3} \d x\, \d y\, \d z\, \Delta_\eps(x,y)[1-\Delta_\eps(y,z)]= \eps\,[1+ \mathrm{O}(\eps^{1/2})].
\end{equation}
Note that when $\Delta_\eps=\eps \Delta$, the constraint reads
\begin{equation}
\label{constr1}
\int_{[0,1]^2} \Delta = 1+\mathrm{O}(\eps^{1/2}),
\qquad \epsilon \downarrow 0.
\end{equation}

The following lemma gives an upper bound for $I_r(1)- I_r(1-\Delta_{\eps})$.

\begin{lemma}{\bf [Order of minimal cost]}
\label{upperbdd:energy}
Let $\Delta_{\eps}: [0,1]^2\to [0,1]$ be such that $1-\Delta_{\eps}\in \mathcal W$ and $\|T_{1-\Delta_{\eps}}\|=1-\eps$. Then, for $\eps$ small enough,
\begin{equation}
I_r(1)- I_r(1- \Delta_{\eps}) \leq \|\Delta_{\eps}\|_1\log \frac{1}{\|\Delta_\eps\|_1} + \mathrm{O}(\|\Delta_\eps\|_1).
\end{equation}
Moreover, $\delta_r(\eps) \leq  \eps\log\frac{1}{\eps}+\mathrm{O}(\eps)$.
\end{lemma}

\begin{proof}
Abbreviate (recall \eqref{Rdef})
\begin{equation}
\chi(a) = \mathcal{R}(a \mid r) = a \log \frac{a}{r} + (1-a) \log \frac{1-a}{1-r}, \quad a \in [0,1].
\end{equation}
Then
\begin{equation}
\begin{aligned}
&\chi(1)- \chi(1-\Delta_\eps(x,y))\\
&= \Delta_{\eps}(x,y)\log\left(\frac{1-\Delta_\eps(x,y)}{\Delta_{\eps}(x,y)}\frac{1-r(x,y)}{r(x,y)}\right)
- \log (1-\Delta_\eps(x,y)),
\end{aligned}
\end{equation}
and so
\begin{equation}
\label{form2}
I_r(1)- I_r(1-\Delta_\eps)
= \int_{[0,1]^2} \left[\Delta_{\eps}\log\left(\frac{1-\Delta_\eps}{\Delta_{\eps}}\frac{1-r}{r}\right)
- \log (1-\Delta_\eps)\right].
\end{equation}

Let $\mu_\eps$ be the probability measure on $[0,1]^2$ whose density with respect to the Lebesgue measure is $Z_{\eps}^{-1}(1-\Delta_{\eps}(x,y))$, where $Z_\eps = \int_{[0,1]^2} (1-\Delta_{\eps}) = 1 - \mathrm{O}(\eps)$. Since $u \mapsto \bar s(u)= u\log(1/u)$ is strictly concave, by Jensen's inequality we have
\begin{equation}
\int_{[0,1]^2}  \Delta_{\eps}\log\left(\frac{1-\Delta_\eps}{\Delta_{\eps}}\right)
= Z_{\eps}\int_{[0,1]^2} \mu_{\eps}\bar{s}\left(\frac{\Delta_{\eps}}{1-\Delta_{\eps}}\right)
\le Z_{\eps}\bar{s}\left(Z_{\eps}^{-1}\|\Delta_\eps\|_1\right)
= \|\Delta_{\eps}\|_1\log \left(\frac{Z_{\eps}}{\|\Delta_{\eps}\|_1}\right).
\end{equation}
Moreover,
\begin{equation}
\int_{[0,1]^2} \Delta_{\eps}\log\left(\frac{1-r}{r}\right) = \mathrm{O}(\|\Delta_{\eps}\|_1),
\qquad - \int_{[0,1]^2} \log (1-\Delta_\eps) = \mathrm{O}(\|\Delta_{\eps}\|_1).
\end{equation}
Hence
\begin{equation}
I_r(1)- I_r(1-\Delta_{\eps}) \leq  \|\Delta_{\eps}\|_1\log \frac{1}{\|\Delta_{\eps}\|_1} 
+ \mathrm{O}(\|\Delta_{\eps}\|_1), \qquad \eps \downarrow 0,
\end{equation}
and since $\|\Delta_{\eps}\|_1 = \mathrm{O}(\eps)$ also $\delta_r(\eps) \leq \eps \log\frac{1}{\eps}+ \mathrm{O}(\eps)$.
\end{proof}

The following is the analogue of Lemma \ref{lem:costperturbation} for perturbations near the right end. 

\begin{lemma}{\bf [Cost of small perturbations]}
\label{lem:costperturbation1}
Let $B \subseteq [0,1]^2$ be a measurable region of area $|B|$. Suppose that $\Delta_{\epsilon} = \epsilon^\alpha\Delta$ on $B$ with $\epsilon > 0$, $\alpha\in (0,1]$ and $\Delta\colon\,[0,1]^2 \to [0,\infty)$. Then the contribution of $B$ to the cost $I_r(h)$ is 
\begin{equation}
\int_B \d x\, \d y\, \Big[\mathcal{R}(1 \mid r(x,y)) - \mathcal{R}(1-\eps^{\alpha} \Delta(x,y) \mid r(x,y))\Big]  
= [1+\mathrm{o}(1)] \int_B  \epsilon^\alpha \Delta\log\left(\frac{1-r}{\epsilon^{\alpha}\Delta r}\right), \quad \epsilon \downarrow 0.
\end{equation}
\end{lemma}

\begin{proof} 
The proof is the same as that of Lemma \ref{upperbdd:energy}, with the observation that
\begin{equation}
\mathcal{R}(1 \mid r) - \mathcal{R}(1- \epsilon^{\alpha}\Delta \mid r)  
= [1+\mathrm{o}(1)]\, \epsilon^{\alpha} \Delta \log\left(\frac{1-r}{\epsilon^{\alpha}\Delta r}\right), \qquad \epsilon \downarrow 0.
\end{equation}
\end{proof}

Following the argument in Section \ref{pertCr}, we can approximate the cost function by using block graphons. We see
\begin{equation}
\label{blockconstraint1}
\sum_{i,j = 1}^N \int_{B_{i,j}} \d x \, \d y \, \Delta_{\epsilon, N}(x,y) 
= \sum_{i,j = 1}^N \frac{1}{N^2}  \overline{\Delta_{\epsilon}}_{N, ij} 
= \epsilon\,[1+\mathrm{o}(1)], \qquad \epsilon \downarrow 0.
\end{equation}
The block constraint in \eqref{blockconstraint1} implies that the sum over each block must be of order $\epsilon$. Hence
\begin{equation}
\overline{\Delta_{\epsilon}}_{N, ij} = \mathrm{O}(\epsilon), \qquad \epsilon \downarrow 0 \qquad \forall\, (i,j).
\end{equation}
There are two cases to distinguish: all blocks contribute to the constraint with a term of order $\epsilon$ (balanced perturbation), or some of the blocks contribute to the constraint with a term of order $\epsilon$ and some with $\mathrm{o}(\epsilon)$. Analogously to the analysis in Section~\ref{pertCr}, by using Lemma \ref{lem:costperturbation1} we can compute the total cost that different types of block perturbations produce. This again shows that the optimal perturbations are the balanced perturbations, consisting of perturbations of order $\epsilon$ on every block. As $N \to \infty$, a sequence of such perturbations converges to a perturbation $\Delta_{\epsilon} = \epsilon \Delta$ with $\Delta\colon\, [0,1]^2 \to [0,\infty)$, which we analyse next.

\begin{lemma}{\bf [Balanced perturbations]}
\label{lem:globalperturbation1}
Suppose that $\Delta_{\epsilon}=\epsilon \Delta$ with $\Delta\colon\,[0,1]^2 \to [0,\infty)$. Then
\begin{equation}
\label{K_1}
\delta_r(\epsilon) = [1+\mathrm{O}(\eps^{1/2})]\,\left\{\eps + \epsilon\log\left(\frac{N^1_r}{\epsilon}\right)\right\}
+ \mathrm{O}(\eps^2), \qquad \epsilon \downarrow 0.
\end{equation}
\end{lemma}

\begin{proof}
By \eqref{constr1} and \eqref{form2},
\begin{equation}
\label{form1}
\begin{aligned}
\delta_r(\epsilon) &=  \sup_{\substack{\Delta\colon\, [0,1]^2\to [0,\infty) \\ 1-\eps\Delta\in \mathcal W \\
\int_{[0,1]^2} \Delta = 1+ \mathrm{O}(\eps^{1/2})} }  
\big[I_r(1) - I_r(1-  \epsilon \Delta)\big] \\
&=  \sup_{\substack{\Delta\colon\, [0,1]^2\to [0,\infty) \\ 1-\eps\Delta\in \mathcal W \\
\int_{[0,1]^2} \Delta = 1+ \mathrm{O}(\eps^{1/2})} }  
\int_{[0,1]^2} \left[\eps \Delta \log\left(\frac{1-\eps \Delta}{\eps \Delta} \frac{1-r}{r}\right)
- \log (1-\eps\Delta)\right].
\end{aligned}
\end{equation}
The integral in \eqref{form1} equals 
\begin{equation}
\int_{[0, 1]^2} \left[\eps \Delta \log\left(\frac{1-r}{\eps\Delta r}\right) - (1-\eps\Delta) \log (1-\eps\Delta)\right]
= \int_{[0,1]^2} \eps \Delta \log\left(\frac{1-r}{\eps\Delta r}\right)
+\eps \int_{[0,1]^2} \Delta + \mathrm{O}(\eps^2).
\end{equation}
Hence
\begin{equation}
\label{deltasup}
\delta_r(\eps) = [1+\mathrm{O}(\eps^{1/2})]\,\left\{\eps +\sup_{\substack{\Delta\colon\, [0,1]^2\to [0,\infty) \\
\int_{[0,1]^2} \Delta= 1} } \int_{[0,1]^2} \epsilon \Delta 
\log\left(\frac{1-r}{\epsilon \Delta r}\right)\right\} + \mathrm{O}(\eps^2),
\end{equation}
where the prefactor arises after we scale $\Delta$ by $1+ \mathrm{O}(\eps^{1/2})$ in order to force $\int_{[0,1]^2} \Delta= 1$. Note that the constraint under the supremum no longer depends on $\eps$.

We can solve the optimisation problem by applying the method of Lagrange multipliers. To that end we define the Lagrangian 
\begin{equation}
\mathcal{L}_{A_r}(\Delta) = \int_{[0,1]^2} \epsilon \Delta \log\left(\frac{1-r}{\epsilon \Delta r}\right)
+ A_r \int_{[0,1]^2} \Delta,
\end{equation}
where $A_r$ is a Langrange multiplier. Since $\int_{[0,1]^2} \log \frac{1-r}{r} = \int_{[0,1]^2} \log \frac{1-r^{\phi}}{r^{\phi}}$ for any Lebesgue measure-preserving bijective map $\phi \in \mathcal{M}$, we get that the minimizer (in the space of functions from $[0,1]^2\to\R$) is of the form
\begin{equation}
\Delta^{\phi} = \e^{-\frac{\epsilon-A_r} \epsilon}\,\frac{1}{\epsilon}\,\frac{1-r^{\phi}}{r^{\phi}}, 
\quad \phi \in \mathcal{M}.
\end{equation}
We pick $A_r$ such that the constraint $\int_{[0,1]^2} \Delta= 1$ is satisfied. This gives
\begin{equation}
\label{minid}
\Delta^{\phi} = \frac{1}{N^1_r}\,\frac{1-r^{\phi}}{r^{\phi}}, \quad \phi \in \mathcal{M},
\end{equation}
with $N^1_r = \int_{[0,1]^2} \frac{1-r}{r}$. Hence the supremum in \eqref{deltasup} becomes 
\begin{equation}
\label{deltasupalt}
\sup_{\phi \in \mathcal{M}} \int_{[0,1]^2} \epsilon \Delta^\phi \log\left(\frac{1-r}{\epsilon \Delta^\phi r}\right).
\end{equation}
We have
\begin{equation}
\int_{[0,1]^2} \epsilon \Delta^\phi \log\left(\frac{1-r}{\epsilon \Delta^\phi r}\right)
= \eps \log\left(\frac{N_r^1}{\eps}\right) - \eps \int_{[0,1]^2} \Delta^\phi \log\left(\frac{\Delta^\phi}{\Delta}\right),
\end{equation}
where we use that $\int_{[0,1]^2} \Delta^\phi = 1$. Since the function $u \mapsto s(u) = u \log u$ is strictly convex on $[0,\infty)$, Jensen's inequality gives 
\begin{equation}
\int_{[0,1]^2} \Delta^\phi \log\left(\frac{\Delta^\phi}{\Delta}\right)
= \int_{[0,1]^2} \Delta\,s\left(\frac{\Delta^\phi}{\Delta}\right) 
\geq s\left(\int_{[0,1]^2} \Delta\,\frac{\Delta^\phi}{\Delta}\right)
= s\left(\int_{[0,1]^2} \Delta^\phi\right) = s(1) = 0,
\end{equation}
where we use that $\int_{[0,1]^2} \Delta = 1$. Equality holds if and only if $\Delta=\Delta^\phi$ almost everywhere on $[0,1]^2$, which amounts to $\phi=\mathrm{Id}$. Hence the supremum in \eqref{deltasupalt} is uniquely attained at $\phi=\mathrm{Id}$ and equals
\begin{equation}
\int_{[0,1]^2} \eps\,\frac{1}{N_r^1} \frac{(1-r)}{r}\log\left(\frac{N_r^1}{\eps}\right) 
= \eps \log \left(\frac{N_r^1}{\eps}\right).
\end{equation}
Consequently, \eqref{deltasup} gives \eqref{K_1}.
\end{proof}

Lemma~\ref{lem:globalperturbation1} settles the claim in Theorem~\ref{thm:rfscaling}(ii). Since we have shown that a balanced perturbation is optimal, \eqref{minid} settles the claim in Theorem~\ref{thm:minscaling}(ii).


\subsection{Perturbation near the left end}
\label{pert0}

Take $\beta = \epsilon$ and consider a graphon of the form $h = \Delta_{\epsilon}$, where $\Delta_{\epsilon} \colon\, [0,1]^2 \to [0,\infty)$ represents a symmetric and bounded perturbation of the constant graphon $h \equiv 0$. We compare 
\begin{equation}
\psi_r(\epsilon) = \inf_{\substack{\Delta_{\epsilon} \colon\, [0,1]^2 \to [0,\infty) \\ \Delta_{\epsilon} \in \mathcal{W} \\\|\Delta_{\epsilon} \|= \epsilon}} 
I_r( \Delta_{\epsilon})
\end{equation}
with 
\begin{equation}
\psi_r(0) = I_r(0)
\end{equation}
by computing the difference
\begin{equation}
\label{around0}
\delta_r(\epsilon) =  \psi_r(\epsilon) - \psi_r(0)
\end{equation}
and studying its behaviour as $\epsilon \to 0$. 

We claim that analysing \eqref{around0} is equivalent to analysing
\begin{equation}
\label{around1}
\delta_{\hat{r}}(\epsilon) =  \phi_{\hat{r}}(1) - \phi_{\hat{r}}(1-\epsilon),
\end{equation}
where $\hat{r}$ is the \emph{reflection} of $r$ defined as
\begin{equation}
\label{sr}
\hat{r}(x,y) = 1-r(x,y), \qquad x,y \in [0,1].
\end{equation}
Indeed,
\begin{equation}
I_r(0) = \int_{[0,1]^2} \mathcal{R}(0 \mid r) = \int_{[0,1]^2} \log\left(\frac{1}{1-r}\right) 
= \int_{[0,1]^2} \mathcal{R}(1 \mid \hat{r}) = I_{\hat{r}}(1) 
\end{equation}
and
\begin{equation}
\begin{split}
I_r(\Delta_\epsilon) &= \int_{[0,1]^2} \mathcal{R}(\Delta_\epsilon \mid r) 
= \int_{[0,1]^2} \left[\Delta_\epsilon\log\left(\frac{\Delta_\epsilon}{r}\right) 
+ (1-\Delta_{\epsilon}) \log\left(\frac{1-\Delta_{\epsilon}}{1-r}\right)\right] \\
&= \int_{[0,1]^2} \mathcal{R}(1- \Delta_{\epsilon} \mid \hat{r}) 
= I_{\hat{r}}(1- \Delta_{\epsilon}).
\end{split}
\end{equation}
We can therefore use the results in Section \ref{sec:pert1}. From Lemma \ref{lem:globalperturbation1} we know that
\begin{equation}
\delta_{\hat{r}}(\epsilon) = [1+\mathrm{O}(\eps^{1/2})]\,\left\{\epsilon 
+ \epsilon \log\left(\frac{N^1_r}{\epsilon}\right)\right\}
+ \mathrm{O}(\eps^2), \qquad \epsilon \downarrow 0,
\end{equation}
and hence we obtain
\begin{equation}
\label{scalingaround0}
\delta_{r}(\epsilon) = [1+\mathrm{O}(\eps^{1/2})]\,\left\{\epsilon 
+ \epsilon \log\left(\frac{N^0_r}{\epsilon}\right)\right\}
+ \mathrm{O}(\eps^2), \qquad \epsilon \downarrow 0,
\end{equation}
Consequently, the optimal perturbation is given by the balanced perturbation $\Delta_{\epsilon} = \epsilon \Delta$ with
\begin{equation}
\label{minidalt}
\Delta = \frac{1}{N^0_r}\, \frac{r}{1-r},
\end{equation}
with $N^0_r = \int_{[0,1]^2} \frac{r}{1-r}$.

The scaling in \eqref{scalingaround0} settles the claim in Theorem~\ref{thm:rfscaling}(iii). Since we have shown that a balanced perturbation is optimal, \eqref{minidalt} settles the claim in Theorem~\ref{thm:minscaling}(iii).


\appendix

\section{Appendix}

\begin{lemma}{\bf [Finite-rank expansion]}
\label{l.finrank}
Consider a graphon $\bar{h} \in \mathcal{W}$ such that 
\begin{equation}
\bar{h}(x,y) = \sum_{i=1}^k\theta_i\bar\nu_i(x)\bar\nu_i(y), \qquad x,y\in[0,1],
\end{equation}
for some $k\in\N$, where $\theta_1>\theta_2\ge\ldots\ge\theta_k \ge 0$ and $\{\bar\nu_1,\bar\nu_2,\ldots,\bar\nu_k\}$ is an orthonormal set in $L^2[0,1]$. Then there exists an $\varepsilon>0$ such that, for any $h \in \mathcal{W}$ satisfying $\| T_{h-\bar{h}} \| < \min(\varepsilon, \|T_h\|)$, the operator norm $\|T_h\|$ solves the equation
\begin{equation}
\|T_h\|=\lambda_k \left(\sum_{n\in\N_0} \|T_h\|^{-n}{\mathcal F}_n(h,\bar h)\right),
\end{equation}
where $\lambda_k(M)$ is the largest eigenvalue of a $k\times k$ Hermitian matrix $M$, and ${\mathcal F}_n(h,\bar h)$ is a $k\times k$ matrix whose $(i,j)$-th entry is 
\begin{equation}
\sqrt{\theta_i\theta_j} \int_{[0,1]^2} \d x \, \d y \, \bar\nu_i(x) (h-\bar h)^n (x,y)\bar\nu_j(y)
\end{equation}
for $1\le i,j\le k$ and $n\in\N_0$.
\end{lemma}

\begin{proof}
Put $\mu=\|T_h\|$, and let $u$ be the eigenfunction of $h$ corresponding to $\mu$, i.e.,
\begin{equation}
T_hu=\mu u.
\end{equation}
Put $g=h-\bar h$ and rewrite the above as
\begin{equation}
(\mu-T_g)u=T_{\bar h} u.
\end{equation}
The assumption $\|T_{h-\bar h}\| < \|T_h\|$ implies that $\mu-T_g$ is invertible, which allows us to write
\begin{equation}
u=(\mu-T_g)^{-1}T_{\bar h} u=\sum_{j=1}^k\theta_j\langle\bar\nu_j,u\rangle (\mu-T_g)^{-1}\bar\nu_j.
\end{equation}
For fixed $1\le i\le k$, it follows that
\begin{equation}
\langle\bar\nu_i,u\rangle = \sum_{j=1}^k\theta_j\langle\bar\nu_j,u\rangle 
\langle\bar\nu_i,(\mu-T_g)^{-1}\bar\nu_j\rangle.
\end{equation}
Multiplying both sides by $\mu\sqrt\theta_i$,  we get
\begin{equation}
\label{l.finrank.eq1}
Mv = \mu v,
\end{equation}
where $M=(M_{ij})_{1 \leq i,j \leq k}$ is the $k\times k$ real symmetric matrix with elements
\begin{equation}
M_{ij} =\sqrt{\theta_i\theta_j}\left\langle \bar\nu_i,\left(1-\frac{T_g}{\mu}\right)^{-1}
\bar\nu_j\right\rangle, \qquad 1\le i,j\le k,
\end{equation}
and
\begin{equation}
v=\left[\sqrt\theta_1 \langle\bar\nu_1,u\rangle,\ldots,\sqrt\theta_k \langle\bar\nu_k,u\rangle\right]^\prime.
\end{equation}
The first entry of $v$ is non-zero for $\varepsilon$ small with $\|T_g\|<\varepsilon$. Thus, \eqref{l.finrank.eq1} means that $\mu$ is an eigenvalue of $M$. By studying the diagonal entries of $M$, we can shown with the help of the Gershgorin circle theorem that, for small $\|T_g\|$,
\begin{equation}
\mu=\lambda_k(M).
\end{equation}
With the help of the observation
\begin{equation}
M_{ij}=\sqrt{\theta_i\theta_j}\sum_{n\in\N_0}
\mu^{-n}\langle\bar\nu_i,g^n\bar\nu_j\rangle,\qquad 1\le i,j\le k,
\end{equation}
i.e.,
\begin{equation}
M = \sum_{n\in\N_0} \mu^{-n}{\mathcal F}_n(h,\bar h),
\end{equation}
this completes the proof.
\end{proof}


\section*{Acknowledgment}
AC and RSH were supported through MATRICS grant of SERB, and FdH and MS through NWO Gravitation Grant NETWORKS 024.002.003. The authors are grateful to ISI and NETWORKS for financial support during various exchange visits to Kolkata and Leiden.



\end{document}